\documentclass{amsart}
\usepackage[utf8]{inputenc}
\usepackage{a4wide}
\usepackage{amsmath,amsfonts,amssymb}
\usepackage{amsthm}
\usepackage{mathtools}
\usepackage{ifthen}
\usepackage{longtable}
\usepackage{array}
\usepackage{url}
\usepackage{enumerate}
\usepackage{enumitem}
\usepackage{xcolor}
\usepackage{float}
\usepackage{graphicx}
\usepackage{todonotes}
\usepackage{hyperref}
\usepackage{cleveref}

\usepackage{titlesec}
\titleformat{\section}
  {\center\normalsize\bfseries}{\thesection.}{1em}{}
\titleformat{\subsection}
  {\normalsize\bfseries}{\thesubsection.}{1ex}{}

\hypersetup{colorlinks,
	linkcolor=green!50!black,
	citecolor=blue!70!black,
	urlcolor=red!50!black
}

\newcommand{\Z}{\mathbb{Z}}
\newcommand{\Q}{\mathbb{Q}}
\newcommand{\R}{\mathbb{R}}

\newcommand{\N}{\mathbb{N}}

\DeclarePairedDelimiter\floor{\lfloor}{\rfloor}
\DeclarePairedDelimiter\abs{\lvert}{\rvert}
\DeclarePairedDelimiter\fract{\lbrace}{\rbrace}
\newcommand{\abst}[1]{\left\| #1 \right\|}

\theoremstyle{definition}
	
	\newtheorem{ex}{Example}
	\newtheorem{probl}{Problem}
	
	\newtheorem{algo}{Algorithm}
	\newtheorem{thm}{Theorem}
    \newtheorem{lem}{Lemma}
    \newtheorem{cor}{Corollary}
    
    \newtheorem{conj}{Conjecture}

    \newtheorem*{claim*}{Claim}
    \newtheorem{rem}{Remark}

\crefname{lem}{lemma}{lemmas}

\begin{document}
\title[Bounds related to $2$LC]{Some Bounds Related to the $2$-adic Littlewood Conjecture}
\subjclass[2020]{11A55, 10A30, 11J13} 
\keywords{Continued fractions, mixed Littlewood conjecture}
\thanks{Research of the second author supported by NSERC grant 2024-03725.
}

\author[D. Vitorino]{Dinis Vitorino}
\address{D. Vitorino,
School of Computer Science
University of Waterloo
Waterloo, ON N2L 3G1
Canada}
\email{dinis@addition.pt}

\author[I. Vukusic]{Ingrid Vukusic}
\address{I. Vukusic,
School of Computer Science
University of Waterloo
Waterloo, ON N2L 3G1
Canada}
\email{ingrid.vukusic\char'100uwaterloo.ca}

\begin{abstract}
For every irrational real $\alpha$, let $M(\alpha) = \sup_{n\geq 1} a_n(\alpha)$ denote the largest partial quotient in its continued fraction expansion (or $\infty$, if unbounded). The $2$-adic Littlewood conjecture (2LC) can be stated as follows: There exists no irrational $\alpha$ such that $M(2^k \alpha)$ is uniformly bounded by a constant $C$ for all $k\geq 0$. In 2016, Badziahin proved (considering a different formulation of 2LC) that if a counterexample exists, then the bound $C$ is at least $8$. We improve this bound to $15$. Then we focus on a ``B-variant'' of 2LC, where we replace $M(\alpha)$ by $B(\alpha) = \limsup_{n\to \infty} a_n(\alpha)$.
In this setting, we prove that if $B(2^k \alpha) \leq C$ for all $k\geq 0$, then $C \geq 5$.
For the proof we use Hurwitz's algorithm for multiplication of continued fractions by 2. 
Along the way, we find families of quadratic irrationals $\alpha$ with the property that for arbitrarily large $K$ there exist $\beta, 2\beta, 4 \beta, \ldots, 2^K \beta$ all equivalent to $\alpha$.
\end{abstract}

\maketitle

\section{Introduction}

Every irrational real number $\alpha$ can be uniquely represented by its infinite simple continued fraction expansion
\[
	\alpha 
	= [a_0; a_1, a_2, \ldots]
	= a_0 + \frac{1}{a_1 + \frac{1}{a_2 + \frac{1}{\ddots}}},
\]
where $a_0$ is an integer and $a_1, a_2, \ldots$ are positive integers.
Continued fractions have many applications in various fields of mathematics, and they have been studied extensively for centuries, see, for example, \cite{RockettSzusz1992} for a reference book.

For $\alpha = [a_0; a_1, a_2, \ldots]$, the integers $a_n$ are called the \textit{partial quotients} of $\alpha$, and we denote the $n$-th partial quotient of any real irrational $x$ by $a_n(x)$.
For every bound $C\geq 1$, it is easy to construct irrational real numbers $\alpha$ with $a_n(\alpha) \leq C$ for all $n \geq 1$. Such numbers are called \textit{badly approximable}.
If we define
\[
	M(\alpha) := \sup_{n \geq 1} a_n(\alpha),
\]
then $\alpha$ being badly approximable is equivalent to $M(\alpha) < \infty$. Badly approximable numbers play an important role in the field of Diophantine approximation, and while they are easy to construct, it is in general a very hard problem to decide whether a given real number is badly approximable or not, and in fact ``most real numbers'' are not.

When multiplying a badly approximable number $\alpha$ by an integer $k$, the partial quotients might increase, but $k \alpha$ will be badly approximable as well.
For example, if $M(\alpha) = C$, then $M(2\alpha) \leq 2C + 3$. This follows, e.g., from \cite[Cor.~4]{Shallit1991}. From \Cref{algo:2algo} in \Cref{sec:algo} one can see that actually $M(2\alpha) \leq 2 M(\alpha) + 2$.
Sometimes, we can even multiply a number by 2 repeatedly without increasing the maximal partial quotients at all. For example, if we set $\alpha = (\sqrt{17}-1)/8$, we have $\alpha = [0; 2, \overline{1, 1, 3}]$, $2\alpha = [\overline{1;3,1}]$, and $4\alpha = [1;\overline{1,1,3}]$, where the bar indicates that the digits repeat periodically.
Now the following might seem like an innocent question:

\begin{probl}\label{probl:$2$LC}
Does there exist a constant $C$ and a real irrational $\alpha$ such that $M(2^k \alpha) \leq C$ for all $k \geq 0$? 
\end{probl}

It turns out that this question is not innocent at all. Any reader familiar with the mixed Littlewood conjecture will know that \Cref{probl:$2$LC} is in fact equivalent to the $2$-adic Littlewood conjecture; see the next section for more background.

We do not prove or disprove the $2$-adic Littlewood conjecture in this paper, but pose and answer questions inspired by \Cref{probl:$2$LC}.

In the next section, we give more background on the relation between the $2$-adic Littlewood conjecture and other problems. We also introduce a variant of \Cref{probl:$2$LC}, where $M(\alpha)$ is replaced with 
\[
	B(\alpha) := \limsup_{n\to \infty} a_n(\alpha).
\]
The smallest number $C$ that can satisfy the property in \Cref{probl:$2$LC} has been studied in \cite{Badziahin2016, Badziahin2025}.
In \Cref{sec:Badziahin}, we slightly improve the (until recently\footnote{Dmitry Badziahin has computed an even better bound \cite{Badziahin2025} since the first version of this paper was published on arXiv.}) best known lower bound on it \cite[Theorem~1.2]{Badziahin2016}.
Moreover, motivated by our proof, we give an equivalent formulation of \Cref{probl:$2$LC}.
In \Cref{sec:cf}, we recall some basic facts about continued fractions.
In \Cref{sec:algo}, we precisely describe an algorithm for multiplying continued fractions by $2$. 
This algorithm is well known; however, we present it in a way which we find more intuitive. This enables us to prove lower bounds for the $B$-variant of \Cref{probl:$2$LC} in \Cref{sec:B-bound}.
Along the way, we find out why digit patterns like $\overline{3,1,1}$ from the example $\alpha = (\sqrt{17}-1)/8$ mentioned above are special. In fact, in \Cref{sec:2equiv} we find an infinite family of digit patterns with the same special property. Our journey also leads to a characterization of all numbers with the property $B(\alpha) \leq 2$ and $B(2\alpha) \leq 2$ (\Cref{sec:B-bound}).
Finally, in \Cref{sec:problems}, we state some open problems.

\section{The $p$-adic Littlewood conjecture and related problems}\label{sec:connections}

Let $\abst{x} = \min_{k \in \Z} \abs{ k - x}$ denote the distance of $x$ to the nearest integer.
As mentioned in the introduction, numbers with the property $M(\alpha) < \infty$ are called badly approximable.
Equivalently, one can define
\[
    m(\alpha) :=  \inf_{q \geq 1}q \abst{q\alpha},
\]
and say that $\alpha$ is badly approximable if and only if $m(\alpha) > 0$. Since this is related to \Cref{probl:$2$LC} and the $2$-adic Littlewood conjecture, let us briefly check the equivalence.
For $\alpha = [a_0;a_1, a_2, \ldots]$ we denote the truncation of $\alpha$ at its n-th partial quotient, $[a_0; a_1, \ldots, a_n]$, by $p_n/q_n$. This is called the $n$-th \textit{convergent} to $\alpha$. 
Convergents are famously particularly good approximations to $\alpha$. In fact, they are ``best approximations'' to $\alpha$ in the sense that if $q_n \leq q < q_{n+1}$ then $\abst{q\alpha} \geq \abst{q_n \alpha} = \abs{q_n \alpha - p_n}$.
Moreover, the following estimates are well known for $n\geq 0$:
\begin{equation}\label{eq:cont_frac_est}
    \frac{1}{(a_{n+1} + 2) q_n}
    < \abst{q_n\alpha}
    < \frac{1}{a_{n+1} q_n}.
\end{equation}
Multiplying the inequalities by $q_n$, using the best approximation property, taking the infimum, and recalling that $M(\alpha) = \sup_{n\geq 1} a_n(\alpha)$, we obtain
\begin{equation}\label{eq:M-m}
    \frac{1}{M(\alpha) + 2}
    \leq m(\alpha)
    \leq \frac{1}{M(\alpha)}.
\end{equation}
In particular, $m(\alpha) > 0$ if and only if $M(\alpha) < \infty$, as desired.

\medskip

The $p$-adic Littlewood conjecture ($p$LC) is a special case of the mixed Littlewood conjecture, which was introduced by de Mathan and Teuli\`{e} in 2004 \cite{DeMathanTeulie2004}.
Let $\abs{n}_p = p^{-v_p(n)}$ denote the $p$-adic absolute value of $n$, where $v_p(n) = \max \{ k \colon p^k \mid n\}$ is the $p$-adic valuation of the integer $n$.

\begin{conj}[$p$LC]\label{conj:pLC}
For every prime $p$ and irrational real $\alpha$, we have
\begin{equation*}
	\inf_{q \geq 1} q \cdot \abs{q}_p \cdot \abst{q \alpha} 
	= 0.
\end{equation*}
\end{conj}

We briefly check that deciding whether \Cref{conj:pLC} is true for $p = 2$ is equivalent to solving \Cref{probl:$2$LC}. We remark that of course we could also formulate \Cref{probl:$2$LC} for general $p$; however, in this paper we choose to focus on the case $p=2$ throughout.

First, note that 
\begin{align}
    \inf_{q \geq 0} q \cdot \abs{q}_2 \cdot \abst{q\alpha} 
    &= \inf_{\substack{k \geq 0,\ q' \geq 0, \\ \gcd(q',2)=1}} 2^k q' \cdot \abs{2^k q'}_2 \cdot \abst{2^k q'\alpha} \nonumber \\
    &= \inf_{\substack{k \geq 0,\ q' \geq 0, \\ \gcd(q',2)=1}} q' \cdot \abst{2^k q'\alpha}
    = \inf_{\substack{k \geq 0 \\ q \geq 1}} q \cdot \abst{q \cdot 2^k \alpha}
    = \inf_{k\geq 0} m(2^k \alpha). \nonumber 
\end{align}
Then we see from inequality~\eqref{eq:M-m} that 
$\inf_{q \geq 0} q \cdot \abs{q}_2 \cdot \abst{q \alpha}  = \inf_{k\geq 0} m(2^k \alpha)> 0$ if and only if $\sup_{k\geq 0} M(2^k \alpha) < \infty$. In other words, we have a counterexample to \Cref{conj:pLC} with $p = 2$ if and only if the answer to \Cref{probl:$2$LC} is ``yes''.
We reformulate \Cref{probl:$2$LC} as a conjecture:
\begin{conj}[$2$LC]\label{conj:$2$LC}
    For every irrational real $\alpha$ we have $\sup_{k\geq 0} M(2^k \alpha) = \infty$.
\end{conj}

For context, let us list a few known facts about $p$LC.
\begin{itemize}
    \item The set of counterexamples has Hausdorff dimension zero \cite{EinsiedlerKleinbock2007}. (For comparison, the set of badly approximable numbers has Lebesgue measure 0 and Hausdorff dimension 1.)

    \item Quadratic irrationals cannot be counterexamples. In fact, Aka and Shapira \cite{AkaShapira2018} showed a much stronger result on the distribution of digit patterns in the continued fraction expansions of $b^k \alpha$ for quadratic irrational $\alpha$.

    \item More generally, if the complexity of the sequence of partial quotients of a real number $\alpha$ grows too rapidly or too slowly, it cannot be a counterexample \cite{BadziahinBugeaudEinsiedlerKleinbock2015}. There also exist recent results on the complexity of the $p$-ary representation of a putative counterexample to $p$LC \cite{BlackmanKristensenNorthey2024}.
\end{itemize}

Moreover, Bazdiahin \cite{Badziahin2016} used a simple algorithm to prove the following theorem:
\begin{thm}[Badziahin, 2016]\label{thm:Badziahin}
For every $\alpha \in \R$, there exists a positive integer $q$ such that
\[
	q \cdot \abs{q}_2 \cdot \abst{q\alpha} < 1/9.
\]
\end{thm}
In other words, for a counterexample $\alpha$ to $2$LC we would have $0 < \inf_{q \geq 1} q \cdot \abs{q}_2 \cdot \abst{q \alpha} < 1/9$.
We will improve this bound to $1/15$ in the next section (\Cref{cor:Badziahin}), using a different method for subdividing the intervals in the algorithm.
In view of inequality \eqref{eq:M-m}, this is of course closely related to proving a lower bound for $\sup_{k\geq 0}M(2^k \alpha)$. 
In fact, we prove $\sup_{k\geq 0}M(2^k \alpha) \geq 15$ in \Cref{thm:M-bound}.
We should mention that since the first version of this paper was published on arXiv, Badziahin \cite{Badziahin2025} has improved the bound $1/9$ in Theorem~\ref{thm:Badziahin} to $1/25$, which is stronger than our bound.

\medskip

Now let us shift our attention from $M(\alpha) = \sup_{n \geq 1} a_n(\alpha)$ to $B(\alpha) = \limsup_{n \to \infty} a_n(\alpha)$.\footnote{To avoid confusing $M$ and $B$, it might help to think of ``$M$'' as standing for ``$\max$''. In a moment, we will also formulate a ``$B$-variant'' of \Cref{conj:$2$LC}, which one could unkindly think of as a ``B movie'' compared to the original $2$LC.}
This is natural because in Diophantine approximation one is typically more interested in ``how well a number is eventually approximable'' rather than in the first finitely many partial quotients. In fact, $B(\alpha)$ is closely related to the quantity
\[
    c(\alpha) := \liminf_{q \to \infty} q \abst{q \alpha},
\]
sometimes referred to as the \textit{Lagrange constant}.
Just like $M$ and $m$ are connected via \eqref{eq:M-m}, we can use inequality~\eqref{eq:cont_frac_est} and the best approximation property of convergents, to see that
\begin{equation}\label{eq:c-B}
    \frac{1}{B(\alpha) + 2}
    =\liminf_{n \to \infty} \frac{1}{a_{n+1} + 2} 
    \leq c(\alpha)
    \leq \liminf_{n \to \infty} \frac{1}{a_{n+1}}
    = \frac{1}{B(\alpha)}.
\end{equation}
    
In \cite[Problem 3.1]{BadziahinBugeaudEinsiedlerKleinbock2015} the authors posed the following problem:
\begin{probl}[Badziahin et al., 2015]\label{probl:BadziahinEtAl}
Prove or disprove that every badly approximable real number $\alpha$ satisfies
\[
    \lim_{n \to \infty} c(n\alpha) = 0.
\]
\end{probl}
If the statement in \Cref{probl:BadziahinEtAl} is true, then it implies $\lim_{k\to \infty} c(2^k \alpha) = 0$. Weakening this statement even further, we would get $\inf_{k\geq 0} c(2^k \alpha) = 0$, and inequality \eqref{eq:c-B} would imply $\sup_{k\geq 0} B(2^k \alpha) = \infty$ for all real irrational $\alpha$.
This is a ``$B$-variant'' of $2$LC! Since $B(x) \leq M(x)$, this would be a stronger statement than $2$LC. 
We are not confident that this statement is true, but for consistency with $2$LC, we formulate it as a conjecture:

\begin{conj}[B$2$LC]\label{conj:B$2$LC}
For every irrational real $\alpha$ we have $\sup_{k\geq 0} B(2^k \alpha) = \infty$.
\end{conj}

As already mentioned, B$2$LC is a priori stronger than $2$LC. Conversely, it might be easier to find a counterexample to B$2$LC than to $2$LC. For example, a number $\alpha$ with the properties $\sup_{n\geq 1} a_n(2^n\alpha) = \infty$ and $a_i(2^k\alpha) \leq C$ for all $i \neq k$ and some constant $C$ would be a B$2$LC counterexample, but not a $2$LC counterexample. We know how to prove that a number of this specific shape cannot exist, but have been unable to prove that B$2$LC and $2$LC are equivalent. In fact, we do not know whether this is true (see \Cref{probl:$2$LC-B$2$LC}).

Just like we do not solve $2$LC, we do not solve B$2$LC in the present paper either. However, as mentioned above, we will 
improve \Cref{thm:Badziahin} by proving $\sup_{k\geq 0} M(2^k \alpha) \geq 15$ (see \Cref{thm:M-bound} in the next section).
Similarly, for B$2$LC we will prove the lower bound $\sup_{k\geq 0} B(2^k \alpha) \geq 5$ in \Cref{thm:B-bound} in \Cref{sec:B-bound}. Proving the $B$-bound will turn out to be more complicated than proving the $M$-bound.
There seems to be a fundamental difference between both conjectures which renders the B-variant less tractable to computer-aided analysis.
This is justified by \Cref{thm:$2$LC-aequ} (\Cref{sec:Badziahin}) and \Cref{thm:B-counterex-m} (\Cref{sec:2equiv}).

\section{Improving the $M$-bound and reformulating $2$LC}\label{sec:Badziahin}
As promised, we start by proving a lower bound on the constant $C$ in \Cref{probl:$2$LC}:
\begin{thm}\label{thm:M-bound}
For every irrational real number $\alpha$ we have $\sup_{k\geq 0} M(2^k\alpha)\geq 15$.
\end{thm}
\begin{proof}
We describe a simple algorithm that can (at least for small $C$) be used to show that there exists no irrational real number $\alpha$ with $\sup_{k\geq 0} M(2^k\alpha)\leq C$. Note that it suffices to consider $\alpha \in (0,1)$ with $1 \leq a_i(\alpha) \leq C$ for $i \geq 1$. 

For every block of digits $a_1 a_2 \cdots a_n$ with $1 \leq a_1, \ldots, a_n \leq C$, let us denote by
\[
	S_{a_1 a_2 \cdots a_n}
	:= \{ \alpha = [0;a_1, \ldots a_n, b_{n+1}, \ldots] \colon 1 \leq b_i \leq C \text{ for all } i \geq n+1 \}
\]
the set of all irrational reals $\alpha \in (0,1)$ with partial quotients bounded by $C$, and whose continued fraction expansion starts with $a_1 a_2 \cdots a_n$.
Our strategy is to exclude more and more prefixes $w$, i.e., to show for more and more $w$ that if $\alpha \in S_w$, then $M(2^k\alpha)> C$ for some $k$.

We start with $n=2$ and the set $A_2 := \{ a_1a_2 \colon 1 \leq a_1, a_2 \leq C \}$ of blocks of length $2$. Here and in general the set $A_n$ will contain all prefixes $w$ of length $n$ that we have not excluded yet.

At each step we set $A_{n+1} := \emptyset$ and then do the following for each $w = a_1a_2\cdots a_n \in A_n$:
\begin{enumerate}
    \item Compute a rational lower bound $\alpha_{\min}$ and a rational upper bound $\alpha_{\max}$ such that $\alpha_{\min} \leq \alpha \leq \alpha_{\max}$ holds for every  $\alpha \in S_w$.
    If $n$ is even, we can set 
    $\alpha_{\min} = [0;a_1,\ldots, a_n, C, 1, C, 1, C+1]$ and $\alpha_{\max} = [0;a_1,\ldots, a_n, 1, C, 1, C+1]$. If $n$ is odd, the same purpose is achieved by switching the two bounds.
    \item\label{it:k} Then for each $k = 1,2, \ldots$ (up to a stopping point) we do the following: Compute the continued fraction expansions of $2^k \alpha_{\min}$ and $2^k \alpha_{\max}$.
    \begin{enumerate}
        \item If the integer parts do not match, i.e., $a_0(2^k \alpha_{\min}) \neq a_0(2^k \alpha_{\max})$, then we have no general information on the continued fraction expansions of $2^k \alpha$ for $\alpha \in S_w$. We \textbf{stop} and do not exclude $w$.
        \item  Otherwise, assume that $a_i(2^k \alpha_{\min}) = a_i(2^k \alpha_{\max})$ for $0 \leq i \leq \ell$, and that $a_{\ell + 1}(2^k \alpha_{\min}) \neq a_{\ell + 1}(2^k \alpha_{\max})$.
        Then $a_i(2^k \alpha) = a_i (2^k \alpha_{\min}) = a_i (2^k \alpha_{\max})$ for $1 \leq i \leq \ell$, and $a_{\ell + 1}(2^k \alpha) \geq \min \{a_{\ell + 1} (2^k \alpha_{\min}), a_{\ell + 1} (2^k \alpha_{\max}) \}$ follow for all $\alpha \in S_w$.
        If one of these values exceeds $C$, we may exclude $w$ and \textbf{stop}. Otherwise, we go to the next $k$.
    \end{enumerate}
\end{enumerate}
Since the integer parts of $2^k \alpha_{\min}$ and $2^k \alpha_{\max}$ will eventually differ, the procedure (\ref{it:k}) stops after finitely many steps, and we have either excluded $w$ or not.
If we have not excluded it, we add all possible extensions of $w = a_1 a_2 \cdots a_n$ to the set $A_{n+1}$, i.e., we put
\[
	A_{n+1} := A_{n+1} \cup \{ a_1 a_2 \cdots a_n a_{n+1} \colon 1 \leq a_{n+1} \leq C \}.
\]
After doing this for all $w \in A_n$, we have produced the next set $A_{n+1}$.
If at some point we end up with an empty set, we must have excluded all $\alpha$, and the procedure terminates.

We have run the algorithm in Sagemath \cite{sagemath}, excluding all $\alpha$ for $C = 14$. This proves the theorem.
\end{proof}

\begin{rem}
Of course, we do not know whether the algorithm in the proof of Theorem~\ref{thm:M-bound} halts for all $C$. In fact, it terminates for a given $C$ if and only if there exists no $2$LC counterexample with constant $C$. This can be seen by an argument analogous to the proof of \Cref{thm:$2$LC-aequ} below.     
\end{rem}

Our result indeed improves Badziahin's bound \cite[Theorem 1.2]{Badziahin2016} from $1/9$ to $1/15$:

\begin{cor}\label{cor:Badziahin}
For every $\alpha \in \R$, there exists a positive integer $q$ such that
\[
	q \cdot \abs{q}_2 \cdot \abst{q\alpha} < 1/15.
\]
\end{cor}
\begin{proof}
For rational $\alpha$ the statement holds trivially.
Let $\alpha$ be irrational. Then by \Cref{thm:M-bound} there exist $k \geq 0$ and $n\geq 1$ such that $a_n^{(k)} := a_n(2^k\alpha)\geq 15$.
Let $q_{n-1}^{(k)}$ be the denominator of the $(n-1)$-st convergent to $2^k \alpha$. Then by \eqref{eq:cont_frac_est}, we have
\[
    \abst{q_{n-1}^{(k)} 2^k \alpha}
    < \frac{1}{q_{n-1}^{(k)} \cdot a_i^{(k)}}
    \leq \frac{1}{q_{n-1}^{(k)} \cdot 15}.
\]
Now set $q = 2^k q_{n-1}^{(k)}$ and note that $\abs{q}_2 \leq 2^{-k}$. 
The above inequality becomes
\[
    \abst{q \alpha} 
    < \frac{1}{q/2^k \cdot 15},
\]
and we obtain
\begin{align*}
    q \cdot \abs{q}_2 \cdot \abst{q \alpha}
    < q \cdot 2^{-k} \cdot \frac{1}{q/2^k \cdot 15} 
    = {1}/{15}.
\end{align*}
\end{proof}

\begin{table}[H]
\center
\resizebox{\textwidth}{!}{\begin{tabular}{c c c c c c c c c c c c c c c}
\hline
	$C$ & 1  & 2  & 3  & 4  & 5  & 6  & 7  & 8  & 9  & 10  & 11  & 12  & 13  & 14\\
\hline
	time & 0.0004  & 0.0007  & 0.0022  & 0.0093  & 0.0368  & 0.0956  & 0.17022  & 1.2867  & 7.8473  & 26.835  & 99.7163  & 445.7725  & 2033.3207  & 7796.3119\\
\hline
	$K$ & 1  & 2  & 4  & 6  & 9  & 16  & 16  & 28  & 37  & 37  & 41  & 47  & 56  & 59\\
\hline
\end{tabular}}
\caption{Computations for the lower bound for $\sup_{k\geq 0} M(2^k \alpha)$; time in seconds.}
\label{tab:M-bound_time}
\end{table}

\begin{rem}\label{rem:M-K}
The computations for the proof of \Cref{thm:M-bound} were done in Sagemath \cite{sagemath} on a MacBook Air with an $8$-core Apple M$3$ processor and $16$GB of RAM. The code is availble at \url{https://github.com/dinisvit/computer_aided_2LC.git}.
The computation time for the case $C=14$ was about two hours. The detailed times for each $C$ are presented in \Cref{tab:M-bound_time}. We additionally keep track of the
largest $k$'s that are used to exclude certain $\alpha$'s. Denoting the largest $k$ used for each $1 \leq C \leq 14$ by $K_C$, we have in fact proved that for every $\alpha$ there exists a $0\leq k \leq K_C$ such that $M(2^k \alpha) > C$.  
\end{rem}

In view of the above remark, it is natural to ask whether some type of converse statement can be made. Indeed, we have the next theorem.

\begin{thm}\label{thm:$2$LC-aequ}
Let $C$ be a constant.
The following two statements are equivalent:
\begin{enumerate}[label = \roman*)]

\item\label{it:$2$LC-ce} There exists a counterexample to $2$LC with constant $C$, i.e., there exists an $\alpha$ with $M(2^k \alpha) \leq C$ for all $k\geq 0$.

\item\label{it:$2$LC-K} For every $K \geq 0$ there exists an irrational $\alpha_K$ such that $M(2^k \alpha_K) \leq C$ for all $0 \leq k \leq K$. 
 
\end{enumerate}
\end{thm}
\begin{proof}
It is clear that \ref{it:$2$LC-ce} implies \ref{it:$2$LC-K}. Now assume that \ref{it:$2$LC-K} holds.
Denote by $\fract{x} = x - \floor{x}$ the fractional part of $x$.
Since the sequence $(\fract{\alpha_K})_{K\geq 0}$ is bounded, it has some accumulation point $\alpha$. Assume for a moment that $\alpha = [a_0;a_1,\ldots,a_n]$ is rational. Then every number that approximates $\alpha$ well, has to be either of the shape $[a_0;a_1, \ldots, a_n, A, b_{n+2} \ldots]$ or $[a_0;a_1, \ldots a_n-1, 1, A, b_{n+2}, \ldots]$ with $A$ very large. Since $M(\alpha_K)\leq C$ for all $K$, there cannot exist a subsequence of $(\fract{\alpha_K})_{K \geq 0}$ converging to $\alpha$. Therefore, $\alpha$ has to be irrational. 
Now we argue that $a_n(2^k \alpha) \leq C$ for all $k\geq 0$ and all $n \geq 1$.
Fix some $n\geq 1$ and $k \geq 0$. Since $\alpha$ is an accumulation point of $(\fract{\alpha_K})_{K\geq 0}$, there exists a $K \geq n\cdot 9^k$ such that $\alpha_K$ and $\alpha$ agree on the first $n\cdot 9^k$ digits. 
As we will see in \Cref{lem:noofdigitsproduced},
when computing the continued fraction expansion of $2x$ from $x$, the first $3 \ell+1$ partial quotients of $x$ fully determine the first $\ell$ partial quotients of $2x$. Therefore, $\alpha_K$ and $\alpha$ agreeing on the first $n\cdot 9^k$ digits implies that $2^k \alpha_K$ and $2^k \alpha$ agree on the first $n$ digits. In particular $a_n(2^k \alpha) = a_n(2^k \alpha_K) \leq C$.
\end{proof}

The equivalence in \Cref{thm:$2$LC-aequ} is not true if we replace $M$ by $B$. This will become apparent from \Cref{thm:B-counterex-m} (\Cref{sec:2equiv}) and \Cref{thm:B-bound} (\Cref{sec:B-bound}).

Therefore, in order to prove a $B$-version of \Cref{thm:M-bound}, we need a different idea. Our approach relies on a good understanding of the algorithm for multiplication by $2$ which we present in \Cref{sec:algo}. Before doing so, let us recall some more basic facts about continued fractions.

\section{Some basic facts about continued fractions}\label{sec:cf}

We briefly summarize all facts that we will use later. For a reference see, for example, \cite{RockettSzusz1992}.

Let $\alpha = [a_0; a_1, a_2, \ldots]$. As mentioned before, the convergents $p_n/q_n = [a_0; a_1, \ldots, a_n]$ are best approximations to $\alpha$ in the sense that if $q_n \leq q < q_{n+1}$ then $\abst{q\alpha} \geq \abst{q_n \alpha} = \abs{q_n \alpha - p_n}$, and they satisfy the inequalities
\[
	\frac{1}{q_{n-1}^2 (a_n+2)} 
	< \left| \alpha - \frac{p_{n-1}}{q_{n-1}} \right|
	< \frac{1}{q_{n-1}^2 a_n}.
\]
Moreover, Legendre's theorem says that they are in fact ``the only really good approximations'': If 
\[
    \left| \alpha - \frac{p}{q} \right|
    < \frac{1}{2q^2},
\]
then $p/q = p_n/q_n$ for some $n$.

The convergents follow the formula $q_0 = 1$, $q_1 = a_1$, and $q_{n+1} = a_{n+1} q_n + q_{n-1}$ for $n \geq 1$.
In particular, this implies that two consecutive denominators $q_n, q_{n+1}$ cannot both be even.

If the continued fraction expansion of $\alpha$ is periodic, i.e., if $a_{n+1}, a_{n+2}, \ldots, a_{n+m}$ keep repeating periodically from there onward, we write 
$\alpha = [0;a_1,\ldots,a_n, \overline{a_{n+1}, \ldots, a_{n+m}}]$. 
The continued fraction expansion of $\alpha$ is periodic if and only if $\alpha$ is a quadratic irrational.
For quadratic irrational $\alpha$ we denote its conjugate by $\overline{\alpha}$.

We say that a continued fraction is \textit{purely periodic} if it is of the shape $[\overline{a_0; a_1, \ldots, a_n}]$. The following characterization is well known:

\begin{lem}\label{lem:purely_periodic}
The continued fraction expansion of a quadratic irrational $\alpha$ is purely periodic if and only if $\alpha >1$ and $-1<\overline{\alpha}<0$.
\end{lem}

Recall that an \textit{algebraic integer} is an algebraic number whose minimal polynomial over the integers is monic. Irrationals which are simultaneously quadratic irrationalities and algebraic integers have a special pattern in their continued fraction expansion. The next lemma also seems to be well known; see for example \cite{VanDerPoorten2004}.

\begin{lem}\label{lem:cf_palindrome}
If $\alpha$ is both a quadratic irrational and an algebraic integer, then its continued fraction expansion is of the shape
\[
    \alpha = [a_0; \overline{a_1, \ldots, a_n} ],
\]
where $a_1 \cdots a_{n-1}$ is a palindrome and $a_n = 2a_0 - (\alpha + \overline{\alpha})$.
\end{lem}

The statement of the next Lemma is mentioned in \cite[p.~84]{Perron1913} and \cite[p.~44]{RockettSzusz1992}. Since we were unable to find it presented as a self-standing lemma, we have included a direct proof.

\begin{lem}\label{lem:Perron}
Let $\alpha$ be a positive quadratic irrational with purely periodic continued fraction expansion $\alpha = [\overline{a_0; a_1, \ldots, a_n}]$ and 
\[
    \alpha_i = [\overline{a_i;a_{i+1}, \ldots, a_n, a_0, \ldots a_{i-1}}]
\]
be its complete quotients for $0 \leq i \leq n$.
Then there exists a $D$ such that each $\alpha_i$ can be represented as
\[
    \alpha_i = \frac{R_i + \sqrt{D}}{S_i}
\]
with integers
\[
    1 \leq R_i \leq \floor{\sqrt{D}},
    \quad
    1 \leq S_i \leq R_i + \floor{\sqrt{D}} \leq 2 \floor{\sqrt{D}},\quad \text{and} \quad D - R_{i}^2 = S_i S_{i-1}
\]
for $0 \leq i \leq n$, where $S_{-1} = (D - R_0^2)/S_0$ is an integer.
\end{lem}
\begin{proof}
Assume that $\alpha$ is the larger of the two roots of some irreducible polynomial $Ax^2 + Bx + C$ with integers $A \geq 1,B,C$.
For $i=0$ we can therefore write
\[
	\alpha_0 
	= \alpha 
	= \frac{-B + \sqrt{B^2 - 4AC}}{2A}
	= \frac{R_0 + \sqrt{D}}{S_0},
\]
where we set $R_0 = -B$, $B^2 - 4AC = D$ and $S_0 = 2A >0$. 
For $\alpha$ purely periodic, we know from \Cref{lem:purely_periodic} that $\alpha>1$ and $-1 < \overline{\alpha} < 0$.
From $0 < \alpha + \overline{\alpha} = 2R_0/S_0$ we get $R_0 \geq 1$.
From $\overline{\alpha} = (R_0 - \sqrt{D})/S_0 < 0$ we get that $R_0 \leq \floor{\sqrt{D}}$. Finally, $\alpha = (R_0 + \sqrt{D})/S_0 >1$ yields $S_0 \leq R_0 + \floor{\sqrt{D}} \leq 2 \floor{\sqrt{D}}$.
Moreover, note that $D - R_0^2 = B^2 - 4AC - B^2 = - 4AC = 2A \cdot (-2C) = S_0 S_{-1}$. Hence, $S_{-1} = -2C$ is indeed an integer.

Now assume inductively that the statements hold for some $i\geq 0$.
We have
\begin{align*}
	\alpha_{i+1} 
	&= \frac{1}{\alpha_i - a_i}
	= \frac{1}{\frac{R_i + \sqrt{D}}{S_i} - a_i}
	= \frac{S_i}{\sqrt{D} - (a_i S_i - R_i)}
	= \frac{S_i(a_i S_i - R_i + \sqrt{D})}{D - (a_i S_i - R_i)^2}.
\end{align*}
Since by induction hypothesis $S_i$ divides $D - R_i^2$, it also divides $D - (a_i S_i - R_i)^2$, and we can set
\begin{align*}
	R_{i+1} = a_i S_i - R_i
	\quad \text{and} \quad
	S_{i+1} = \frac{D - (a_i S_i - R_i)^2}{S_i}.
\end{align*}
We also get $D - R_{i+1}^2 = D - (a_i S_i - R_i)^2 = S_{i+1} S_i$. 
Since $\alpha_{i+1}$ is purely periodic, the estimates for $R_{i+1}, S_{i+1}$ hold by the same arguments as for $R_0, S_0$.
\end{proof}

Finally, let us recall the concept of equivalence. Two real irrationals $\alpha, \beta$ are said to be \textit{equivalent} (write $\alpha \sim \beta$) if there exist integers $a,b,c,d$ with $ad - bc = \pm 1$ such that
\[
    \beta = \frac{a\alpha + b}{c\alpha + d}.
\]
It turns out that $\alpha \sim \beta$ if and only if their continued fraction expansions have the same tail, that is, if $\alpha = [a_0;a_1, \ldots, a_n, c_1, c_2, \ldots]$ and $\beta = [b_0;b_1,\ldots, b_m, c_1,c_2, \ldots]$.

We will later be interested in determining whether $2\alpha$ can be equivalent to $\alpha$.
In this context, the next lemma will be useful.

\begin{lem}\label{lem:equiv_mipo}
Let $m\neq 0$ be an integer and $\alpha$ be a real irrational. If $\alpha \sim m\alpha$, then $\alpha$ must be a quadratic irrational.
Moreover, if $Ax^2 + Bx + C$ is the minimal polynomial of $\alpha$ (i.e., $A>0$ and $\gcd(A,B,C) = 1$), then $\alpha \sim m\alpha$ if and only if 
there exist integers $a,b,c,d, \ell$ with $\gcd(\ell,m)=1$ and
\begin{align}
    A\ell &= mc, \label{eq:A}\\
    B\ell &= md - a,\label{eq:B}\\
    C\ell &= - b, \label{eq:C}\\
    ad - bc &= \pm 1. \label{eq:abcd}
\end{align}
\end{lem}
\begin{proof}
Assume that $\alpha \sim m\alpha$. Then by definition there exist integers
$a,b,c,d$ with $ad - bc = \pm 1$ and
\begin{equation}\label{eq:k-equiv}
    m \alpha  = \frac{a\alpha + b}{c\alpha + d}.
\end{equation}
This transforms into
\begin{equation}\label{eq:k-mipo}
    mc\alpha^2 + (md - a) \alpha - b = 0,
\end{equation}
i.e., $\alpha$ is a root of the polynomial $mcX^2 + (md - a) X - b$. Since $\alpha$ is irrational, it must therefore be a quadratic irrational.  Moreover, the polynomial has to be a multiple of its minimal polynomial, which means that equations \eqref{eq:A}--\eqref{eq:C} must be satisfied for some integer $\ell \neq 0$.
Further, we must have $\gcd(\ell,m)=1$ because otherwise we see from \eqref{eq:B} and \eqref{eq:C} that $\gcd(\ell,m)$ is a common divisor of $a$ and $b$, contradicting \eqref{eq:abcd}.

Conversely, if $\alpha$ has minimal polynomial $Ax^2 + Bx + C$ and equations \eqref{eq:A}--\eqref{eq:abcd} are satisfied by some integers, then clearly \eqref{eq:k-mipo} and therefore \eqref{eq:k-equiv} holds.
\end{proof}

\section{The algorithm for multiplication by 2}\label{sec:algo}

It is well known how to multiply continued fractions by integers.
More generally, given the continued fraction expansion of $\alpha$, there exist algorithms for computing the continued fraction expansion of $(a\alpha + b)/(c\alpha + d)$ for integers $a,b,c,d$ with $ad - bc \neq 0$. For example, Raney's algorithm \cite{Raney1973} is based on matrix multiplication and automata.
An older algorithm which works only when $B(\alpha)\geq 2|ad-bc|$ is given in \cite{Hurwitz1896}.
Mendès France \cite{Mendes1976} and Gosper \cite{GosperContFracArith} also give algorithms for computing the continued fraction expansion of $(a\alpha + b)/(c\alpha + d)$.
The original analyses of the these two algorithms, however, only justify termination\footnote{In the sense that for any $n$ it is possible to determine the $n$-th partial quotient of $(a\alpha + b)/(c\alpha + d)$ in finite time.} when $\alpha$ is rational.

For multiplication by two, there is a particularly simple and direct way to think about it, dating back to Hurwitz's 1891 paper \cite{Hurwitz1891}. There, he stated the following identities. 
Note that here nonnegative real numbers are allowed in the finite continued fractions.

\begin{lem}\label{lem:Hurwitz}
For integers $m,a,b\geq 0$ and a real number $\theta > 1$, we have
\begin{align*}
    2 \cdot [a;2m,b,\theta] &= [2a; m, 2b, \theta/2],\\
    2 \cdot [a;2m+1,\theta] &= [2a; m, 1, 1, (\theta-1)/2].
\end{align*}
\end{lem}
\begin{proof}
The identities can be verified with simple algebra. 
\end{proof}

On the one hand, this shows that multiplication by $2$ is closely related to division by $2$ and computing $(x-1)/2$.
On the other hand, we can use the identities from \Cref{lem:Hurwitz} to obtain an algorithm for multiplication by 2. First, note that we can write $[2b;\theta/2] = 2b + 2\cdot [0;\theta]$. Second, for $\theta = [b;\omega]$ it is easy to check that $[1;(\theta - 1)/2] = 1 + 2\cdot [0; b-1, \omega]$.
Thus, \Cref{algo:2algo} gives us the rules (cf.\ \cite[p.~360, exercise 14]{Knuth1998})
\begin{align*}
    2 \cdot [0;2m,b,c, \ldots] &= [0; m, 2b + 2 \cdot [0;c,\ldots]],\\
    2 \cdot [0;2m+1, b,c, \ldots] &= [0; m, 1, 1 + 2 \cdot [0; b-1, c, \ldots]].
\end{align*}
In practice, this gives us the following algorithm.

\begin{algo}\label{algo:2algo} 
Given $\alpha = [a_0;a_1,a_2,\ldots]$, we want to compute the continued fraction expansion of $2\alpha$.
First, we record the digit $2a_0$ for the integer part of the resulting number (this might change later, after a ``clean-up process'').
Then we start by looking at the ``window'' $(a_1,a_2,a_3) =: (a,b,c)$. 
In each step, we record two or three digits and then ``slide to the next window'', according to the following rules:

\medskip

\textit{If $a$ is even:}
\begin{itemize}
    \item  Record the digits $a/2$ and $2b$.
    \item Slide to the next block $(c,d,e)$. I.e., if we were considering $(a_n, a_{n+1}, a_{n+2})$ or $(a_n -1, a_{n+1}, a_{n+2})$, then consider $(a_{n+2}, a_{n+3}, a_{n+4}) =: (a,b,c)$ next.
\end{itemize}

\textit{If $a$ is odd:}
\begin{itemize}
    \item Record the digits $(a-1)/2$, $1$ and $1$.
    \item Slide to the next block $(b-1,c,d)$. I.e., if we were considering $(a_n, a_{n+1}, a_{n+2})$ or $(a_n -1, a_{n+1}, a_{n+2})$, then consider $(a_{n+1}-1, a_{n+2}, a_{n+3}) =: (a,b,c)$ next.
\end{itemize}

This way we record a sequence of digits $d_1, d_2, \ldots$, where $d_i \geq 0$. Note that some of the $d_i$'s can be zero (coming from the cases where $a = 0$ or $a = 1$).
In order to eliminate the zeros, we do a ``clean-up'' by going from left to right and replacing every block of the shape $x,0,y$ by the single digit $x+y$.
This gives the continued fraction expansion of $2\alpha$. 
It is worth noting that the algorithm indeed provides us with the first $m$ partial quotients of $2\alpha$ after examining a finite number of partial quotients of $\alpha$. For precise bounds, see \Cref{lem:noofdigitsproduced}.
\end{algo}

Let us consider an example, which will be relevant again later.

\begin{ex}\label{ex:311}
Let $\alpha = (3 + \sqrt{17})/2 =[\overline{3;1,1}]$. When multiplying $\alpha$ by two, we start by recording the digit $6$ for the integer part. 
Then we consider the block $(1,1,3)$. Since its first digit, $1$, is odd, we record $0,1,1$, and move on to $(1-1,3,1)=(0,3,1)$. 
Now the first digit, $0$, is even, and so we record $0, 6$, moving on to $(1,1,3)$. Now we are back to where we started, which means that the process repeats periodically and we record the digits $0,1,1,0,6,0,1,1,0,6, \ldots$.
Together with the digit for the integer part, we have obtained $2\alpha = [6;0,1,1,0,6,0,1,1,0,6, \ldots]$.
Now we do the clean-up.
At the beginning, $6,0,1$ cleans up to $7$. After that, every block $1,0,6,0,1$ cleans up to $8$. Thus, we obtain $2\alpha = [7;\overline{8}]$.
\end{ex}

\begin{rem}\label{rem:windows}
Let $(a_n, a_{n+1}, a_{n+2})$ be a triple of consecutive partial quotients of $\alpha$. Then there are exactly three possibilities for how the algorithm will ``slide across'' this window:

\begin{enumerate}[label = \textit{Case \arabic*:}, ref= \arabic*]
\item The window $(a,b,c) = (a_n, a_{n+1}, a_{n+2})$ is considered because $n=1$ or because the first number of the previously considered block, corresponding to $(a_{n-2},a_{n-1},a_n)$, was even.

\item The modified window $(a,b,c) = (a_n - 1, a_{n+1}, a_{n+2})$ is considered because the first number of the previously considered block, corresponding to $(a_{n-1}, a_n, a_{n+1})$, was odd.

\item The window $(a_n, a_{n+1}, a_{n+2})$ is ``skipped'' because the first number of the previously considered block, corresponding to $(a_{n-1}, a_n, a_{n+1})$, was even. This led to recording $2 a_n$ and moving on to $(a_{n+1}, a_{n+2}, a_{n+3})$.
\end{enumerate}
\end{rem}

\begin{lem}\label{lem:parities}
Let $\alpha = [0; a_1, a_2, \ldots]$ and assume that we apply \Cref{algo:2algo} for multiplication by $2$ to $\alpha$. 
Let $n \geq 2$. Then the parities of $q_{n-1}, q_{n-2}$ determine in which of the three cases from \Cref{rem:windows} the window $(a_n, a_{n+1}, a_{n+2})$ is reached:

\begin{enumerate}[label = \textit{Case \arabic*:}, ref= \arabic*]
\item\label{it:case0} The block $(a_n, a_{n+1}, a_{n+2})$ is considered if and only if $q_{n-2}$ is even (recall that this implies that $q_{n-1}$ is odd).

\item\label{it:case-1} The block $(a_n - 1, a_{n+1}, a_{n+2})$ is considered if and only if $q_{n-2}$ and $q_{n-1}$ are both odd.

\item\label{it:caseskip} The window $(a_n, a_{n+1}, a_{n+2})$ is ``skipped'' (i.e., $2 a_{n}$ was recorded) if and only if $q_{n-1}$ is even (and $q_{n-2}$ is necessarily odd).
\end{enumerate}
\end{lem}
\begin{proof}
Recalling that $q_0 = 1$, $q_1 = a_1$, and $q_{n+1} = a_{n+1} q_n + q_{n-1}$ for $n\geq 1$, the lemma follows from a straighforward induction.
\end{proof}

\begin{rem}\label{rem:proof_parities}
In \Cref{lem:parities} we see that each partial quotient $a_n$ either gets roughly divided by $2$ (Cases~\ref{it:case0} and \ref{it:case-1}), or multiplied by $2$ (Case~\ref{it:caseskip}). In fact, $a_n$ gets multiplied by two if and only if $q_{n-1}$ is even. This can also be seen from the following argument:
We know that
\[
	\frac{1}{q_{n-1}^2 (a_n+2)} 
	< \left| \alpha - \frac{p_{n-1}}{q_{n-1}} \right|
	< \frac{1}{q_{n-1}^2 a_n}.
\]

Now if $q_{n-1}$ is even, we can rewrite this as 
\[
	\frac{1}{(q_{n-1}/2)^2 \cdot (2a_n+4)} 
	< \left| 2\alpha - \frac{p_{n-1}}{q_{n-1}/2} \right|
	< \frac{1}{(q_{n-1}/2)^2 \cdot 2a_n}.
\]
By Legendre's theorem, $p_{n-1}/(q_{n-1}/2)$ is a convergent to $2 \alpha$, and from the inequalities we see that the corresponding partial quotient is between $2 a_n - 2$ and $2a_n + 4$, i.e., roughly of the size $2 a_n$.

On the other hand, if $q_{n-1}$ is odd, then still 
\[
    \left| 2\alpha - \frac{2 p_{n-1}}{q_{n-1}} \right|
	< \frac{1}{(q_{n-1})^2 \cdot (a_n/2)}.
\]
So if we assume $a_n \geq 4$, then by Legendre's Theorem $q_{n-1}$ still corresponds to a convergent to $2 \alpha$, but now the partial quotient is roughly of size $a_n/2$.
\end{rem}

As mentioned earlier, multiplication by 2, division by 2, and computing $(x-1)/2$ are closely related. 
In fact, all three operations can essentially be done with the same algorithm, except that the three variants of the algorithm correspond to the three different ways of reaching windows from \Cref{lem:parities}. Note that $(x-1)/2$ and $(x+1)/2$ differ by $1$, and addition of an integer to a continued fraction alters only its zeroth partial quotient. Hence, we might as well consider $(x+1)/2$ instead. This is more natural if we think in terms of computations modulo 1 because there $x/2$ is in fact the set $\{x/2, (x+1)/2\}$. The next lemma describes what happens when we compute $x/2$ and $(x+1)/2$ compared to computing $2x$.

\begin{lem}\label{lem:notcollide}
Let $\alpha = [a_0; a_1, a_2, \ldots]$ and let 
\begin{align*}
	[b_0; b_1, b_2, \ldots] &= 2 \cdot [0;a_0,a_1, \ldots] = 2\cdot 1/\alpha,\\
	[c_0; c_1, c_2, \ldots] &= 2 \cdot [0;a_0 + 1, a_1,a_2, \ldots] = 2 \cdot 1/(\alpha + 1).
\end{align*}
Then we have
\begin{align*}
	\alpha/2 &= [0;b_0,b_1,\ldots],\\
	(\alpha+1)/2 &= [0;c_0,c_1,c_2, \ldots],
\end{align*}
where the digits $b_0, c_0$ might be zero.
Moreover, as we apply \Cref{algo:2algo} to $\alpha$, $1/\alpha$ and $1/(\alpha +1)$, the considered windows never ``collide'', i.e. for every $n\geq 1$ the three algorithms reach the window $(a_n, a_{n+1}, a_{n+2})$ in exactly the three different cases from \Cref{lem:parities}.
\end{lem}
\begin{proof}
The identities hold because $1/[a_0;a_1,a_2, \ldots] = [0;a_0,a_1, \ldots]$.

The fact the algorithms do not ``collide'' can be checked by a straightforward induction. Another way to see this is the following: Assume that there exists an $\alpha = [a_0;a_1,a_2, \ldots]$ such that two of the three algorithms collide at a window $(a_n, a_{n+1}, a_{n+2})$.
Since the behaviour of the ``sliding'' of the algorithms only depends on the parities of $a_1, a_2, \ldots$, the same would happen for example for the number $\alpha' := [a_0, a_1 + 2, a_2 + 4, a_3 + 6, \ldots]$. Thus, we may assume w.l.o.g.\ that $\alpha$ has unbounded partial quotients. Now, since two of the algorithms collide at the window $(a_n, a_{n+1}, a_{n+2})$, from then on they produce the same digits. But then two of the numbers $2\alpha, \alpha/2, (\alpha + 1)/2$ are equivalent. This implies that $\alpha$ satisfies a quadratic equation, which contradicts our assumption that $\alpha$ has unbounded partial quotients.
\end{proof}

In \Cref{ex:311}, for every block $1,1,3$ that we passed, we recorded a total of $3 + 2 = 5$ digits. In the clean-up process, blocks of length $5$ were reduced to single digits. This means that overall, we had to read three digits of $\alpha$ to produce one digit of $2\alpha$. This is indeed the ``worst'' possible situation. The ``best'' case is if every digit produces three digits, which happens if $a_1$ is odd and $a_2, a_3, \ldots$ are all even and $\geq 4$. 
We formulate this as a lemma.

\begin{lem}\label{lem:noofdigitsproduced}
Let $\alpha=[a_0;a_1,a_2,\dots]$ and suppose \Cref{algo:2algo} is given its first $n+1$ partial quotients $a_0, a_1, \ldots, a_{n}$.
Moreover, assume that every recorded $0$ is cleaned up as soon as possible (i.e., as soon as the succeeding digit is recorded).
Then the algorithm produces a sequence of digits $b_0, b_1, \ldots, b_{m}, b_{m+1}$, where $b_0,\dots,b_{m}$ are the first $m+1$ partial quotients of $2\alpha$ (and $b_{m+1}$ may still be affected by the clean-up), and we have
\[
    \floor{(n+2)/3} - 1
    \leq m
    \leq 3n - 1.
\]
These two bounds are tight in the following sense: For each bound there exist numbers $\alpha$ such that the algorithm achieves the bound for all $n$.
\end{lem}
\begin{proof}
The upper bound $3n$ results from the following observations:
\begin{itemize}
    \item In the initialization phase, the zeroth partial quotient leads to the recording of exactly one digit. After that, the first window $(a_1,a_2,a_3)$ is considered in the first step.
    \item In each step, the considered window slides to the right by at least one index.
    \item The algorithm records at most three new digits per step.
\end{itemize}

Hence, when given the partial quotients $a_0, a_1, \ldots, a_n$, the algorithm takes at most $n$ steps, recording at most $1+3n$ digits.
The last digit may still be affected by the clean-up, resulting in the bound $m+1 \leq 3n$.
This upper bound is attained on any $n$ and $\alpha$ of the form $[a_0;a_1,a_2,\dots]$ where $a_0$ is an arbitrary integer, $a_1\geq 3$ is odd, and $a_i\geq 4$ is even for all $i\geq 2$.


To prove the lower bound, we need to account for the clean-up process. Let us have a closer look at what the algorithm does if we do the clean-up immediately. We consider four cases, according to the type of block $(a,b,c)$ that is hit in each step.
\begin{enumerate}[label = \alph*)]
    \item $a\geq 3$ odd. We record $(a-1)/2,1,1$, and no clean-up happens. This means that we consumed one digit and produced three digits.
    \item $a\geq 2$ even. We record $a/2, 2b$, and no clean-up happens. We consumed 2 digits and produced two digits.
    \item\label{it:1} $a = 1$. We record $0,1,1$. In the clean-up, the first two recorded digits ``disappear'', in the sense that they get added to the the previously produced digit. Thus, we have produced one digit (the last recorded $1$) and consumed one digit.
    \item\label{it:0} $a = 0$. We record $0,2b$. In the clean-up, both digits ``disappear'', in the sense that they get added to the previously produced digit. We have produced zero digits and consumed two.
\end{enumerate}
Clearly, Case~\ref{it:0} produces the least number of digits, namely it produces zero and consumes two. However, after Case~\ref{it:0} we move on to $(c, d , e)$, where $c\neq 0$. Therefore we cannot be in Case~\ref{it:0} for two consecutive steps.
In all other cases, at least one digit is produced per consumed digit.
In the initial step, one digit is produced, and in the first step we cannot be in 
Case~\ref{it:0}.
Overall, the worst possible case is that we record one initial digit, and then alternate between Case~\ref{it:1} and Case~\ref{it:0}. This happens exactly for numbers of the shape $[x_0;1,1,x_1,1,1,x_2,1,1,\dots]$, and one can check that after using the first $n+1$ digits $a_0, a_1, \ldots, a_{n}$, the algorithm produces the digits $b_0, b_1, \ldots, b_{m+1}$ with $m+1 = \floor{(n+2)/3}$.
\end{proof}

\section{Multiplication by 2 and equivalence}\label{sec:2equiv}

\begin{ex}\label{ex:311-2}
Recall that for $\alpha = (3 + \sqrt{17})/2 =[\overline{3;1,1}]$ from \Cref{ex:311} we have $2\alpha = [6;\overline{8}]$.
However, if one computes the related numbers $\alpha/2$ and $(\alpha+1)/2$, something interesting happens: We get $\alpha/2 = [\overline{1;1,3}]$ and $(\alpha+1)/2 = [2;\overline{3,1,1}]$; in other words, $\alpha \sim \alpha/2 \sim (\alpha+1)/2$.

Moreover, we can compute, for example, $(\frac{\alpha}{2} + 1)/2 = [1;2,\overline{1;1,3}]$, so this number is equivalent to $\alpha$ as well.
\end{ex}

Thinking in terms of the three cases of the algorithm, it is perhaps already clear what is going on in \Cref{ex:311-2}: For every $\alpha \sim [\overline{3;1,1}]$, two of the three numbers $2\alpha, \alpha/2, (\alpha+1)/2$ are equivalent to $[\overline{3;1,1}]$ as well. In particular, this allows us to construct arbitrarily long ``chains'' of numbers of the shape $\alpha \sim 2\alpha \sim \dots \sim 2^k \alpha \sim [\overline{3;1,1}]$.

To formalize this, we state the next two Lemmas.

\begin{lem}\label{lem:2outof3}
Assume that $\alpha \sim \alpha/2 \sim (\alpha+1)/2$. Then for each $\beta \sim \alpha$ exactly two out of the three numbers 
\[
    2\beta,\ \frac{\beta}{2},\ \frac{\beta+1}{2}
\]
are equivalent to $\alpha$.    
\end{lem}
\begin{proof}
First of all, note that by \Cref{lem:alpham_equiv}, $\alpha\sim\frac{\alpha}{2}$ implies that $\alpha$ is a quadratic irrational. Assume that the period in its continued fraction expansion is $\overline{c_1, \ldots , c_n}$.
By \Cref{lem:notcollide}, the two algorithms for computing $\alpha/2$ and $(\alpha + 1)/2$ hit the first occurrence of the period in two out of the three possible ways from \Cref{rem:windows} (namely $(c_1, c_2, c_3)$ or $(c_1 - 1, c_2, c_3)$ or skip that window). 
On the other hand, since $\beta \sim \alpha$, the algorithms for computing 
$2\beta, \beta/2, (\beta+1)/2$ will also eventually hit the period $\overline{c_1, \ldots ,c_n}$. Two out of the three algorithms will from that point on produce the same digits as the algorithms for computing $\alpha/2$ and $(\alpha + 1)/2$. 
Hence, two elements of the set $\{2\beta,\ \frac{\beta}{2},\ \frac{\beta+1}{2}\}$ must be equivalent to $\alpha$.

Finally, we check that not all three numbers can be equivalent to $\alpha$. 
Assume towards a contradiction that they are.
This means that all three ways to hit the periodic part of $\beta$ result in a number that is equivalent to $\beta$. By extension, we get $\beta\sim 2^k\beta$ for all $k\in\N$. Letting $Ax^2+Bx+C$ be the minimal polynomial of $\beta$, \Cref{lem:equiv_mipo} implies that for each $k\in\N$, there exist $c_k,\ell_k\in\Z$ such that $A\ell_k=2^kc_k$ and $\gcd(\ell_k,2)=1$. This leads to a contradiction as soon as $k$ is so large that $A$ is not divisible by $2^k$.
\end{proof}

Now we can build arbitrarily long ``chains'' of equivalent numbers.

\begin{lem}\label{lem:long_equ_blocks}
Assume that $\alpha \sim \alpha/2 \sim (\alpha+1)/2$ for some quadratic irrational $\alpha$. 
Then for every $K$ there exists a $\beta$ such that $\alpha \sim \beta \sim 2 \beta \sim \dots \sim 2^K \beta$.
\end{lem}
\begin{proof}
Assume that $\alpha \sim \alpha/2 \sim (\alpha+1)/2$ and set $\beta_K = \alpha$. By assumption, at least one of $\beta_K/2, (\beta_K + 1)/2$ is equivalent to $\alpha$. We set $\beta_{K-1}$ to be such a number, i.e., $\beta_{K-1} \sim \alpha$ and either $2 \beta_{K-1} = \beta_K$ or $2 \beta_{K-1} = \beta_K + 1$. Analogously, \Cref{lem:2outof3} enables us to construct $\beta_{K-2} \sim \alpha$ from $\beta_{K-1}$. Inductively, we can construct $\beta_1 \sim \beta_2 \sim \dots \sim \beta_{K-1} \sim \beta_{K} \sim \alpha$ with the property that $2 \beta_\ell \in \{ \beta_{\ell+1}, \beta_{\ell+1}+1\}$ for every $1\leq \ell \leq K-1$. This implies that $\beta_1 \cdot 2^\ell$ and $\beta_\ell$ only differ by an integer. Thus, $\beta_1  =: \beta \sim 2\beta \sim \dots \sim 2^K \beta \sim \alpha$.
\end{proof}

\begin{ex}\label{ex:311-chain}
In particular, from \Cref{lem:long_equ_blocks} and \Cref{ex:311-2}, we see that for every $K\geq 0$ it is possible to find an $\alpha$ such that $2^k \alpha \sim [\overline{3;1,1}]$ for all $0 \leq k \leq K$.
\end{ex}

An obvious question is whether other digit patterns with this property exist. The next two lemmas provide an infinite family of such patterns.
In fact, the family is a direct generalization of the example $[\overline{3;1,1}]$, which is a root of $x^2 - 3x - 2$.
We remark that this family does not cover all patterns with this property. We currently do not know how to characterize all such patterns; see \Cref{sec:problems} for more details and open problems.

\begin{lem}\label{lem:alpham_equiv}
Let $m\geq 3$ be an odd integer and let $\alpha$ be a root of the polynomial $x^2 - mx - 2$.
Then $\alpha \sim \alpha/2 \sim (\alpha+1)/2$.    
\end{lem}
\begin{proof}
The equation $\alpha^2 - m\alpha - 2 = 0$ implies
\[
    \frac{\alpha}{2} = \frac{1}{\alpha-m} = \frac{0 \cdot \alpha + 1}{1 \cdot \alpha - m}
    \quad \text{and} \quad
    \frac{\alpha + 1}{2} = \frac{\frac{m+1}{2} \alpha + 1 }{1 \cdot \alpha + 0},
\]
and since $m$ is odd, we have $\alpha/2 \sim \alpha$ and $(\alpha+1)/2 \sim \alpha$ by definition.
\end{proof}

\begin{lem}\label{lem:alpham_shape}
Let $m\geq 3$ be an odd integer and let $\alpha$ be the positive root of the polynomial $x^2 - mx - 2$, i.e.,
\[
    \alpha = \frac{m + \sqrt{m^2 + 8}}{2}.
\]
Then the continued fraction expansion of $\alpha$ is of the shape
\[
    \alpha = [\overline{m; a_1, \ldots, a_n}],
\]
where $a_1 \cdots a_n$ is a palindrome and $a_i \leq m$ for $1 \leq i \leq n$.
\end{lem}
\begin{proof}
First, note that for $m=3$ it is easy to check that $\alpha = [\overline{3;1,1}]$, which is of the correct shape. Now assume $m\geq 5$. Then since $m^2 + 8$ cannot be a square, $\alpha$ is irrational. Since it is by definition an algebraic integer, \Cref{lem:cf_palindrome} gives us the shape $\alpha = [a_0; \overline{a_1, \ldots, a_n}]$, where $a_1 \cdots a_{n-1}$ is a palindrome and $a_n = 2a_0 - (\alpha + \overline{\alpha})$.
Moreover, we have $m < \sqrt{m^2 + 8} < m+1$, and so $a_0 = \floor{\alpha} = m$ and $2a_0 - (\alpha + \overline{\alpha}) = 2m - m = m$. Thus, $\alpha$ is of the correct shape.
We only need to show that $a_i \leq m$ for $1 \leq i \leq n-1$. 

Assume towards a contradiction that $a_i \geq m+1$ for some $1 \leq i \leq n-1$. By \Cref{lem:Perron} we have
\[
    \alpha_i 
    = [a_i; a_{i+1}, \ldots, a_{n-1}, m, a_1, \ldots a_{i-1}]
    = \frac{R_i + \sqrt{m^2 + 8}}{S_i}
\]
for some integers $0 \leq R_i \leq m$ and $1 \leq S_i \leq 2m$.
But then
\[
    m+1  
    < \alpha_i 
    \leq \frac{m + \sqrt{m^2 + 8}}{S_i}
    < \frac{2m +1}{S_i},
\]
which implies $S_i = 1$. 
Since $\alpha_i$ is purely periodic as well, we have $-1 < \overline{\alpha_i} < 0$ by  \Cref{lem:purely_periodic}. This yields $R_i=m$, i.e,
\[
    \alpha_i 
    =m + \sqrt{m^2 + 8}
    = 2 \alpha.
\]
In particular, this implies $\alpha \sim 2\alpha$, which is impossible by \Cref{lem:2outof3} as we already have $\alpha \sim \alpha/2 \sim (\alpha + 1)/2$ from \Cref{lem:alpham_equiv}.
\end{proof}

Now we immediately get the next theorem.

\begin{thm}\label{thm:B-counterex-m}
For every odd $m\geq 3$ and for every $K\geq 0$ there exists an $\alpha_{m,K}$ such that $B(2^k \alpha_{m,K}) = m$ holds for all $0\leq k \leq K$.
\end{thm}
\begin{proof}
We can set $\alpha_{m,0}$ to be the positive root of $x^2 - mx - 2$. By \Cref{lem:alpham_shape} we have $B(\alpha_{m,0}) = m$, and by Lemmas~\ref{lem:alpham_equiv} and \ref{lem:long_equ_blocks} for every $K$
we can construct a suitable $\alpha_{m,K}$ from $\alpha_{m,0}$.
\end{proof}

\begin{rem}
Recall that in the proof of \Cref{thm:M-bound} we actually did the following (see \Cref{rem:M-K}): For each $1 \leq C \leq 14$, we found a $K_C$ so that we were able to prove that for every $\alpha$ there exists a $0\leq k \leq K_C$ such that $M(2^k \alpha) > C$. 
In view of \Cref{thm:B-counterex-m}, we will not be able to achieve this in the $B$-setting for $C\geq 3$.
\end{rem}

\section{Proving the $B$-bound}\label{sec:B-bound}

In \Cref{thm:M-bound} we proved the lower bound $\sup_{k\geq 0} M(2^k \alpha) > 14$.
The goal of the present section is to prove a similar bound $\sup_{k\geq 0} B(2^k \alpha) > C$, with $C$ as large as we can manage. Since the approach used to prove the $M$-bound does not transfer to $B$, we require a more intricate strategy.

We start by proving the smallest lower bound $\sup_{k\geq 0} B(2^k \alpha) > 1$.

\begin{lem}\label{lem:1}
If $\alpha \sim [0;\overline{1}]$, then $2\alpha \sim [0;\overline{4}]$.
\end{lem}
\begin{proof}
From \Cref{rem:windows} and \Cref{lem:notcollide} we know that if $\alpha \sim [0;\overline{1}] = (\sqrt{5}-1)/2 =:\beta$, then $2\alpha$ is equivalent to one of $2\beta, \beta/2, (\beta+1)/2$. It is easy to check that all three numbers are equivalent to $[0;\overline{4}] = \beta/2$.
\end{proof}

\Cref{lem:1} implies that there exists no $\alpha$ with $B(\alpha) \leq 1$ and $B(2\alpha) \leq 1$. In particular, there cannot exist a B$2$LC counterexample with constant $C=1$.

If we set $C=2$, however, this is slightly less obvious. Indeed, we have
\[
    2 \cdot [0;1,\overline{2}] = 2 \cdot \frac{\sqrt{2}}{2} = [1;\overline{2}].
\]
In view of this example, it is natural to ask whether there exist other numbers $\alpha$ different from  $\sqrt{2}/2$ which have the property $B(\alpha) \leq 2$ and $B(2\alpha) \leq 2$.
It turns out that there exist infinitely many such numbers. However, each such number is in one of two equivalence classes. They are quadratic irrationals, and we are able to characterize them. We remark that we could skip this step on the way to proving $\sup_{k\geq 0} B(2^k \alpha) > 2$. However, the next theorem may be interesting in its own right.

\begin{thm}
\label{theo:characterB2}
Let $\alpha$ be an irrational real number. Then $B(\alpha) \leq 2$ and $B(2\alpha) \leq 2$ if and only if the continued fraction of $\alpha$ is of one of the following two shapes:
\begin{align}
    \alpha &= [a_0; a_1, \ldots, a_n, \overline{2}] 
        \quad \text{and} \quad
        q_{n} \equiv q_{n-1} \equiv 1 \pmod{2}, \label{eq:B2-2}\\
    \alpha &= [a_0; a_1, \ldots, a_n, \overline{2,1}] 
        \quad \text{and} \quad
        q_{n-1} \equiv 0 \pmod{2}. \label{eq:B2-21}
\end{align}
\end{thm}
\begin{proof}
Let $\alpha = [a_0;a_1, a_2, \ldots]$ and $2\alpha = [b_0; b_1, b_2, \ldots]$, and assume that $B(\alpha) \leq 2$ and $B(2\alpha) \leq 2$. Then there exists an index $M$ such that $a_i, b_i \leq 2$ for all $i \geq  M$. We investigate what happens as our algorithm for multiplication by two ``slides across'' $a_{3M+1}, a_{3M+2}, \ldots$. 
By \Cref{lem:noofdigitsproduced}, the digits produced from  $a_{3M+1}, a_{3M+2}, \ldots$ correspond to digits $b_i$ with $i \geq M$. Thus, these must be bounded by $2$.

We now systematically look at all blocks $(a,b,c)$, which might be considered in the process. We assume that the digit $c$ is followed by $d,e$. Note that since $a_i \in \{1,2\}$ for all $i \geq M$, by the rules of the algorithm, we always have $a \in \{0,1,2\}$ and $b,c,d,e \in \{1,2\}$.

\begin{enumerate}[label = \textit{Case \arabic*:}, ref= \arabic*]

\item\label{c:0} $(a,b,c) = (0,b,c)$. We record the digits $0$ and $2b$, and move on to $(c, d, e)$. Since $2b\geq 2$, we obtain a digit $\geq 3$ after cleaning up the produced zero, a contradiction.

\item\label{c:11} $(a,b,c) = (1,1,c)$. We record the digits $0$, $1$ and $1$, and move on to $(0,c,d)$. But then we get a contradiction in the next step (cf., Case~\ref{c:0}).

\item\label{c:121} $(a,b,c) = (1,2,1)$. We record the digits $0$, $1$ and $1$, and move on to $(1,1,d)$. This leads us to land in Case~\ref{c:11} which again culminates in a contradiction.

\item\label{c:122} $(a,b,c) = (1,2,2)$. We record the digits $0$, $1$ and $1$, and move on to $(1,2,d)$. To avoid the contradiction in Case~\ref{c:121}, the next digit $d$ must be a $2$. Hence, in the next step we are again in the same situation that $(a,b,c) = (1,2,2)$. Using this argument inductively, we see that all digits after this point must be equal to 2. Thus, $a_i = 2$ for all $i$ sufficiently large. Regarding the output, we record the digits $\overline{0,1,1,0,1,1}$ which after the clean-up become $0,1,\overline{2}$.

\item\label{c:211} $(a,b,c) = (2,1,1)$. We record the digits $1$ and $2$, and move on to $(1, d, e)$. This means we end up in one of the Cases \ref{c:11}, \ref{c:121}, \ref{c:122}. Cases~\ref{c:11} and \ref{c:121} yield a contradiction. Case~\ref{c:122} gives $a_i = 2$ for all $i$ sufficiently large.

\item\label{c:22} $(a,b,c) = (2,2,c)$. We record the digits $1$ and $4$, and move on to $(c, d, e)$. Thus, we have recorded the digit $4$, a contradiction.

\item\label{c:212} $(a,b,c) = (2,1,2)$. We record the digits $1$ and $2$, and move on to $(2,d,e)$. This means we continue with one of the Cases~\ref{c:211}, \ref{c:22}, \ref{c:212}. Case~\ref{c:22} yields a contradiction; Case~\ref{c:211} either yields a contradiction or $a_i = 2$ for all $i$ sufficiently large. This means that we either get a contradiction, or $a_i = 2$ for all $i$ sufficiently large, or we keep coming back to Case~\ref{c:212}. The latter implies that for $i$ sufficiently large, the digits $a_i$ follow the pattern $\overline{2,1}$. Regarding the output, we record the digits $\overline{1,2}$.

\end{enumerate}

Overall, we see that we either get a contradiction, or eventually stay in Case~\ref{c:122}, or eventually stay in Case~\ref{c:212}. 

Staying in Case~\ref{c:122} corresponds to $\alpha$ and $2\alpha$ being equivalent to $[0;\overline{2}]$. Moreover, if we have $\alpha = [a_0; a_1, \ldots, a_n, \overline{2}]$, the algorithm has to reach the block $(a_{n+1}-1, a_{n+2},a_{n+3}) = (1,2,2)$. By \Cref{lem:parities}, $q_n$ and $q_{n-1}$ must be odd.

Saying we eventually stay in Case~\ref{c:212} amounts to saying that $\alpha$ and $2\alpha$ are equivalent to $[0;\overline{2,1}]$.
If $\alpha = [a_0; a_1, \ldots, a_n, \overline{2,1}]$, we have to reach the block $(a_{n+1}, a_{n+2},a_{n+3}) = (2,1,2)$. By \Cref{lem:parities}, $q_n$ must be odd and $q_{n-1}$ must be even.

Thus, we have proven that if $B(\alpha)\leq 2$ and $B(2\alpha)\leq 2$, then $\alpha$ must be of one of the two shapes \eqref{eq:B2-2}, \eqref{eq:B2-21}. 
Our arguments and \Cref{lem:parities} show that the converse holds as well.
\end{proof}

Next, we prove that we cannot have $B(\alpha)\leq 2$ and $B(4\alpha) \leq 2$.

\begin{lem}\label{lem:B2}
For every irrational real $\alpha$ we have $B(\alpha/2) \geq 3$ or $B(2\alpha) \geq 3$ (or both).
\end{lem}
\begin{proof}
Let $\alpha = [a_0; a_1, a_2, \ldots]$.
Note that if $B(\alpha) = 1$, then $\alpha \sim [0;\overline{1}]$, and we are done by \Cref{lem:1}. Therefore, we may assume $B(\alpha)\geq 2$.

By \Cref{lem:notcollide}, we can compute the digits of $\alpha/2$ and $2\alpha$ by applying \Cref{algo:2algo} to $1/\alpha$ and $\alpha$, respectively. 

Aiming towards a contradiction, assume that $B(\alpha/2)\leq 2$ and $B(2\alpha)\leq 2$. This means that for sufficiently large $n$, as we slide across $(a_n,a_{n+1},a_{n+2})$, both algorithms must produce digits $\leq 2$. Since $B(\alpha)\geq 2$, there exist arbitrarily large indices $n$ such that the window $(a,b,c)=(a_n,a_{n+1},a_{n+2})$ satisfies $a = a_n \geq 2$.
From \Cref{lem:notcollide}, we know that our two algorithms will reach the window $(a,b,c)$ in two distinct ways.

If one of the two algorithms skips $(a,b,c)$, it must have recorded $2a\geq 4$ in the previous step, a contradiction.
Therefore, one of them must reach $(a,b,c)$, and the other one reaches $(a-1,b,c)$. 
Since either $a$ or $a-1$ is even, one of the two algorithms will record $\floor{a/2}$ and $2b$. This implies $b = 1$.
Moreover, either $a$ or $a-1$ is odd. Hence, the other algorithm will record $\floor{(a-1)/2}, 1, 1$ and move on to $(b-1,c,d) = (0,c,d)$. Then it will record $0$ and $2c\geq 2$. After cleaning up the produced $0$, this becomes $\geq 3$, a contradiction.
\end{proof}

In particular, \Cref{lem:B2} implies that there does not exist a B$2$LC counterexample with constant $2$. We proceed to exclude the constant $3$.

\begin{lem}\label{lem:B3}
Let $\alpha$ be a real irrational with $B(4 \alpha) = 3$ and $4\alpha \not\sim [\overline{3;1,1}]$. Then $B(2^{k} \alpha) \geq 4$ for at least one of $k \in \{0,1,3,4\}$.
\end{lem}
\begin{proof}
Let $4\alpha = [a_0; a_1, a_2, \ldots]$ with $B(4\alpha) = 3$, and assume that $B(2^k \alpha) \leq 3$ for $k \in \{0,1,2,3,4\}$.
Then, if $n$ is sufficiently large, all digits in $2^{k} \alpha$ for $k \in \{0,1,2,3,4\}$ that are tied to $a_n$ in terms of the algorithms have to be bounded by $3$.
Since $B(4\alpha) = 3$, there exist arbitrarily large indices $n$ such that $(a_n, a_{n+1}, a_{n+2}) = (3,b,c)$.
As in the proof of \Cref{lem:B2}, we know that the algorithms for division and multiplication by $2$ do not ``collide'' at $a_n$. Therefore, when computing $2\alpha$ and $8 \alpha$ from $4\alpha$, exactly two out of the three cases from \Cref{lem:parities} occur.

If either one of the algorithms skips $(3,b,c)$, it recorded $2\cdot 3 = 6$  in the previous step, which is impossible. 
Therefore, the two algorithms reach $(3,b,c)$ in exactly the two following ways:
\begin{itemize}
    \item Algorithm A: We arrive at $(2,b,c)$, record $1, 2b$, and move on to $(c, d, e)$. Since the produced digits are at most $3$, we must have $b = 1$.
    \item Algorithm B: We arrive at $(3,b,c) = (3,1,c)$, record $1,1,1$, and move on to $(0,c,d)$.
    In the next step, we record $0, 2c$ and move on to $(d, e, f)$. Say we record a digit $x$ in the following step. Then taking zeroes which must be cleaned up into consideration, we see that we must have $c = 1$ and $x \geq 1$.
\end{itemize}

Note that we have proven that at each $a_n = 3$ with $n$ sufficiently large, we must have $(a_n, a_{n+1}, a_{n+2}) = (3,1,1)$ in order to avoid producing digits $\geq 4$. 
By extension, the same holds for the digits of $2\alpha$ and $8\alpha$. In other words, in the expansions of $2\alpha, 4\alpha$ and $8\alpha$ eventually every digit $3$ must be followed by $1,1$.

Now let us go back to Algorithm A and see what else it produces.
We have recorded $1,2 = 2b$, and moved on to $(c,d,e) = (1,d,e)$. Then we record $0,1,1$, and move on to $(d-1,e,f)$.
Assume for a moment that $d = 1$ or $d = 2$. In both cases we record a $0$ after that, which means that overall, we record $1,2,0,1,1,0,*$.
After the clean-up, we obtain $1,3,y$ with $y \geq 2$. But now the digits computed by Algorithm A violate the rule that every $3$ in the expansions of $2\alpha$ and $8\alpha$ must be followed by two $1$'s. Therefore, we must have $d = 3$.

We have proven that whenever a digit $3$ appears in $4\alpha$ at a sufficiently large index, it must be followed by $1,1,3$. This means that $4\alpha \sim [\overline{3;1,1}]$.
\end{proof}

We can push this argument one step further.

\begin{lem}\label{lem:B4}
Let $\alpha$ be a real irrational with $B(4\alpha) = 4$. Then $B(2^k \alpha) \geq 5$ for at least one of $k \in \{0,1,3,4\}$.
\end{lem}
\begin{proof}
Let $\alpha = [a_0; a_1, a_2, \ldots]$ with $B(4\alpha) = 4$.
As in the proof of \Cref{lem:B3}, assume that for $n$ sufficiently large, every digit $a_n = 4$ is only tied to digits $\leq 4$.

Then, as in the proofs of \Cref{lem:B2,lem:B3}, we know that the algorithms for division and multiplication by $2$ applied to $4\alpha$ reach the window $(a_n, a_{n+1},a_{n+2}) = (4,b,c)$ in two distinct ways and that neither skips this window.
Therefore, the two algorithms reach $(4,b,c)$ in exactly the two following ways (we write $z = a_{n-1}$):
\begin{itemize}
    \item Algorithm A: We record $*, 2z$, arrive at $(4,b,c)$, record $2, 2b$, and move on to $(c, d, e)$. This implies 
    \[
        z \leq 2 \quad\text{and}\quad b \leq 2.
    \]
    \item Algorithm B: We arrive at $(3,b,c)$ after recording $*, 1, 1$. This means that in the previous block, $(z,4,b)$ or $(z-1,4,b)$, the first entry must have been odd. Since $z \leq 2$, it must have been $1$. This means that we actually recorded $0, 1,1$ before arriving at $(3,b,c)$. After that, we record $1,1,1$ and move on to $(b-1,c,d)$.
    Since $b-1\leq 1$, we next record $0, *$. Overall, we have recorded $0, 1, 1, 1, 1, 1, 0, *$, so after clean-up, we obtain the pattern $u,1,1,1,w$ with $u \geq 2$ and $w \geq 2$.
\end{itemize}

We have proven that the pattern $u,1,1,1,w$ with $u \geq 2$ and $w \geq 2$ must occur in either $2\alpha$ or $8\alpha$. We now look at what the algorithms for multiplication and division by $2$ do to this pattern.

If we skip the window $(1,1,1)$, we arrive at $(1,1,w)$. We record $0,1,1$ and move to $(0,w,*)$. Then we record $0,2w$, which leads to a contradiction because $w \geq 2$.

Thus, the other two cases must occur. In particular, in one case we arrive at the block $(1,1,1)$ after having recorded $2u$. Then we record $0,1,1$, and move on. But after the clean-up, we get a digit $\geq 2u + 1$. This is a contradiction because $u \geq 2$.
\end{proof}

We have been unable to increase the lower bound on $C$ any further and therefore collect our results:

\begin{thm}\label{thm:B-bound}
For all irrational reals $\alpha$ we have
\[
    \sup_{k \geq 0} B(2^k \alpha) \geq 5.
\]
\end{thm}
\begin{proof}
For $4 \alpha \not\sim [\overline{3;1,1}]$, this follows from Lemmas \ref{lem:B2}, \ref{lem:B3} and \ref{lem:B4}.
Now assume $4 \alpha \sim [\overline{3;1,1}]$.
As in the proof of \Cref{lem:2outof3}, we cannot have $2^k \alpha \sim [\overline{3;1,1}]$ for all $k\geq 2$.
Hence, there must exist some $k_0 \geq 2$ with $2^{k_0}\alpha \sim [\overline{3;1,1}]$ and $2^{k_0+1}\alpha \not\sim [\overline{3;1,1}]$. It follows from \Cref{ex:311}, \Cref{ex:311-2} and \Cref{lem:notcollide} that in fact $B(2^{k_0+1} \alpha) = 8$.
\end{proof}

\begin{rem}\label{rem:5}
We have not been able to push the bound in \Cref{thm:B-bound} beyond $5$.
The issue is the following: Both $\alpha = [\overline{5,2,1,2}]$ and $\alpha = [\overline{5;1,2,2,1}]$ have the property $\alpha \sim \alpha/2 \sim (\alpha+1)/2$. 
Thus, by \Cref{lem:long_equ_blocks}, it can happen that $\alpha \sim 2\alpha \sim \dots \sim 2^K \alpha \sim [\overline{5,2,1,2}]$ or $\sim [\overline{5;1,2,2,1}]$ for arbitrarily large $K$.
But then, if we want to use the same type of arguments as in the proof of \Cref{lem:B3}, we have to account for these special cases, which seems rather unpleasant. In fact, we know from \Cref{thm:B-counterex-m} that we will run into this type of trouble every time we try to push beyond an odd number.
\end{rem}

\section{Comments and open problems}\label{sec:problems}

The number $[\overline{5;2,1,2}]$ from \Cref{rem:5} corresponds to 
$m = 5$ in \Cref{lem:alpham_shape}, that is, to the positive root of $x^2 - 5x - 2$.
On the other hand, the number $[\overline{5;1,2,2,1}]$ from \Cref{rem:5} does not correspond to an $\alpha$ from \Cref{lem:alpham_shape}. One can check that it is a root of $x^2 - 5x - 4$.
In fact, it is easy to find many other equivalence classes containing $\alpha$'s with the property $\alpha \sim \alpha/2 \sim (\alpha+1)/2$.

In a previous version of this paper, we pointed out that all examples we found had a similar digit structure: They were of the shape $\alpha \sim [\overline{m; a_1, \ldots, a_n}]$, where $m$ is an odd integer, $a_1 \cdots a_n$ is a palindrome, and $a_i \leq (m-1)/2$ for all $i$. We suggested that this might always be the case. However, in a personal communication, Menny Aka provided us with counterexamples. For instance, $\alpha = (1 + \sqrt{2089})/6$ has the property $\alpha \sim \alpha/2 \sim (\alpha+1)/2$, but the period is $14, 1, 3, 1, 1, 1, 2, 1, 6, 1, 8, 3, 1, 2, 3, 2, 4, 7, 2, 1, 1, 4, 2$.
This seem to be related to the following facts: the prime $2$ splits in the number field $\Q(\sqrt{2089})$, the two ideals lying above $2$ are principal, and the ideal class of $(2, \alpha)$ has odd order $>1$.

Moreover, Menny Aka commented that the following generalization of Lemma~\ref{lem:long_equ_blocks} 
and our examples can be proved using degenerate branches from \cite[Definition 4.7]{AkaShapira2018}: For every $K\in \N$ and every prime $p$, one can find a quadratic irrational $\beta$ with all the quadratic irrationals $\beta, p \beta,…, p^K \beta$ having the same periodic part.

It seems that with the right tools from algebraic number theory and dynamical systems, one should be able to solve the next problem.

\begin{probl}\label{probl:equ}
Characterize all equivalence classes that contain an element with the property $\alpha \sim \alpha/2 \sim (\alpha+1)/2$.
\end{probl}

A solution to Problem~\ref{probl:equ} might, however, not reveal much about shape of the periodic parts of such numbers.
From Lemmas~\ref{lem:alpham_equiv} and \ref{lem:alpham_shape} we know that all odd numbers $\geq 3$ are realized as maxima of the periodic parts.
But what about the period length? Does the set of irrational reals $\alpha$ with $\alpha \sim \alpha/2 \sim (\alpha+1)/2$ contain a quadratic irrationality whose periodic part has length $\ell$ for every $\ell\geq 3$? It is easy to check that when $\ell\in\{1,2\}$, no such numbers exists. Additionally, for $\ell=3$ it turns out that all such $\alpha$ are equivalent to $[\overline{3;1,1}]$.
Therefore, the following problem may be of interest:

\begin{probl}
    Let $\ell$ be a positive integer. How many inequivalent quadratic irrationalities $\alpha$ with period length $\ell$ and $\alpha\sim \alpha/2 \sim (\alpha+1)/2$ exist?
\end{probl}

Going back to the conjectures $2$LC and B$2$LC, they are of course still open, and one could try to further improve the lower bound for $\sup_{k \geq 0} M(2^k \alpha)$ and our bound $\sup_{k \geq 0} B(2^k \alpha) \geq 5$. Perhaps a solution to \Cref{probl:equ} might help with the latter.

\begin{probl}
    Improve the bound $\sup_{k \geq 0} B(2^k \alpha) \geq 5$.
\end{probl}

As mentioned in \Cref{sec:connections}, we do not know whether B$2$LC and $2$LC are equivalent. For example, it might be possible to prove that if there exists a B$2$LC counterexample, then there also exists a $2$LC counterexample (possibly with a different constant).

\begin{probl}\label{probl:$2$LC-B$2$LC}
    Are B$2$LC (\Cref{conj:B$2$LC}) and $2$LC (\Cref{conj:$2$LC}) equivalent?
\end{probl}

Finally, we are not aware of any metric results on B$2$LC.
\begin{probl}\label{probl:HD}
    Does the set of counterexamples to B$2$LC (\Cref{conj:B$2$LC}) have Hausdorff dimension 0?
\end{probl}

\section{Acknowledgements}

We are very grateful to Menny Aka for his comments on the first version of this paper.
Moreover, I.V.\ wants to thank Jeffrey Shallit for many helpful discussions and for pointing her to Hurwitz's algorithm. She also wants to thank Manuel Hauke for long and helpful discussions and for suggesting to consider $B$ instead of $M$. 

\bibliographystyle{habbrv}
\bibliography{refs}

\end{document}